\documentclass[11pt]{amsart}
\usepackage{graphicx}
\usepackage{hyperref}
\usepackage{palatino}
\usepackage[margin=2cm]{geometry}
\usepackage[english]{babel}

\newtheorem{theorem}{Theorem}[section]
\newtheorem{lemma}[theorem]{Lemma}

\theoremstyle{definition}

\theoremstyle{remark}

\numberwithin{equation}{section}
\allowdisplaybreaks
\begin{document}

\title[Convex ordering of P\'{o}lya random variables and error monotonicity of Bernstein-Stancu operators]{Convex ordering of P\'{o}lya random variables and monotonicity of the error estimate of Bernstein-Stancu operators}
\author[D. Mele\c{s}teu]{Diana Mele\c{s}teu}
\address{Faculty of Mathematics and Computer Science, Transilvania University of Bra\c{s}ov, Str. Iuliu Maniu 50, Bra\c{s}ov -- 500091, Romania.}
\email{diana.melesteu@unitbv.ro}

\author[M. N. Pascu]{Mihai N. Pascu}
\address{Faculty of Mathematics and Computer Science, Transilvania University of Bra\c{s}ov, Str. Iuliu Maniu 50, Bra\c{s}ov -- 500091, Romania.}
\email{mihai.pascu@unitbv.ro}

\author[N. R. Pascu]{Nicolae R. Pascu}
\address{Department of Mathematics, Southern Polytechnic State University, 1100 South Marietta Parkway, Marietta, GA 30060-2896, U.S.A.}
\email{npascu@spsu.edu}


\begin{abstract}
In the present paper we show that in P\'{o}lya's urn model, for an arbitrarily fixed initial distribution of the urn, the corresponding random variables satisfy a convex ordering with respect to the replacement parameter. As an application, we show that in the class of convex functions, the absolute value of the error of Bernstein-Stancu operators is a non-decreasing (strictly increasing under an additional hypothesis) function of the corresponding parameter.

The proof relies on two results of independent interest: an interlacing lemma of three sets and the monotonicity of the (partial) first moment of P\'{o}lya random variables with respect to the replacement parameter.
\end{abstract}

\maketitle

%% Start line numbering here if you want % \linenumbers

%% main text
%%%%%%%%%%%%%%%%%%%%%%%%%%%%%%%%%%%%%%%%%%%%%%%%%%%%%%%%%%%%%%%%%%%%%%%%%%%%%%%%%%%%%%

\section{Introduction}

More than $100$ years ago, in a beautiful and short paper Serge Bernstein (%
\cite{Berstein 1912}) gave a simple, constructive proof of Weierstrass's
theorem on uniform approximation of con\-ti\-nu\-ous functions by polynomials,
known nowadays as Bernstein polynomials.

About $50$ years later, D. D. Stancu noticed that the binomial distribution
used by Bernstein is a particular case of the P\'{o}lya's urn distribution
(the case of the replacement parameter being equal to zero), and he
introduced (\cite{Stancu'68Japan}, \cite{Stancu'69}) a more general class of
polynomials/operators, known in the literature as the P\'{o}lya-Stancu
operators (the operator $P_{n}^{c}$ defined by (\ref{Bernstein-Stancu
operator})).

Aside from a passing remark that for a particular choice (negative real
number) of the replacement parameter $c$ one obtains the Lagrange
interpolation polynomial (which cannot be used for uniform approximation),
in his work Stancu considered only non-negative values $c\geq 0$ of the
replacement parameter in P\'{o}lya's urn model, this choice being adopted by subsequent researchers in the field.

Recently (in \cite{PaPaTr1}, see also \cite{PaPaTr2} and \cite{TrPa}), the
last two authors introduced the operator $R_{n}$ (given by (\ref{Pascu operator})),
corresponding to a negative choice (pointwise minimal) of the replacement
parameter of P\'{o}lya's urn distribution, and showed that this leads to
better approximation results. To be precise, we showed that the upper bounds
of the error estimates (in terms of the first/second order modulus of
continuity of the function, etc) are smaller than the corresponding estimates for the
Bernstein operator, and we also provided numerical evidence (for various
choices of the function, smooth, continuous, even discontinuous - see \cite%
{PaPaTr1}) which indicated that among all Bernstein-Stancu type operators,
the newly introduced operator $R_{n}$ gives the best approximation.

A criticism received while publishing these results was that even
though the upper bounds for the error of the new operator are smaller than
the corresponding ones for the Bernstein operator, this is not a proof that
the new operator is a better approximation operator.

In the present paper, we fill this gap (at least partially), by showing that
for a sufficiently large class of functions (convex functions) the absolute
value of the error of approximation of Bernstein-Stancu type operators $%
P_{n}^{c}$ is (pointwise) a monotonically increasing function of the
replacement parameter $c$. This shows that in the case of operators
considered by Stancu (non-negative replacement parameter), the Bernstein
operator $B_{n}$ gives the best approximation of convex functions, and that
the choice of the operator $R_{n}$ considered by the last two authors (minimal
admissible choice of the replacement parameter) further improves this approximation, giving the best
approximation of convex functions.

The proof of the main result (Theorem \ref{Theorem on monotonicty of error})
relies on a result of independent interest (Theorem \ref{Theorem on Polya
convex ordering}), which shows that the P\'{o}lya random variables satisfy a
convex ordering with respect to the replacement parameter. In turn, the
proof of this result relies on two other results with interest on their own:
a result concerning the ordering (interlacing) of three sets (Lemma \ref%
{interlacing lemma}) and the monotonicity of the partial centered moment of
the P\'{o}lya distribution (Lemma \ref{monotonicity of Polya first centered
moment}).

The structure of the paper is the following. In Section \ref{Section
Preliminaries} we introduce the notation and the main results needed in the
sequel.

In Section \ref{Section Main results} we prove the convex ordering of P\'{o}%
lya random variables, and, as an application, in Section \ref{Section
Applications} we prove that in the case of convex functions, the absolute
value of the error of approximation of the Bernstein-Stancu operators is a
non-decreasing function of the replacement parameter (strictly increasing under an
additional hypothesis).

\section{Preliminary results\label{Section Preliminaries}}

Recall that for given integer parameters $a,b\ge0$, $c$, and $n\geq 1$
satisfying the compati\-bi\-lity condition
\begin{equation}
a+\left( n-1\right) c\geq 0\text{\qquad and\qquad }b+\left( n-1\right) c\geq
0\text{,}  \label{hypotheses on a,b,c}
\end{equation}%
P\'{o}lya's urn model gives the number of white balls ("successes")
extracted in $n$ trials from an urn containing initially $a$ white balls and
$b$ black balls, where after each extraction, the extracted ball is replaced
in the urn together with $c$ balls of the same color. Denoting by $%
X_{n}^{a,b,c}$ the random variable representing the number of successes in
this experiment (which will be referred to as a P\'{o}lya random variable
with parameters $a,b,c$ and $n$), we have
\begin{equation}  \label{Polya probabilities}
p_{n,k}^{a,b,c}=P\left( X_{n}^{a,b,c}=k\right) =C_{n}^{k}\frac{a^{\left(
k,c\right) }b^{\left( n-k,c\right) }}{(a+b)^{\left( n,c\right) }}, \qquad
k\in \{0,1,\ldots,n\},
\end{equation}
where for $x,h\in \mathbb{R}$ and $n\in\mathbb{N}$ we denoted by
\begin{equation}
x^{\left( n,h\right) }=x\left( x+h\right) \left( x+2h\right) \cdot \ldots
\cdot \left( x+\left( n-1\right) h\right)  \label{rising factorial}
\end{equation}%
the generalized (rising) factorial with increment $h$. We are using the
convention that an empty product is equal to $1$, that is $x^{\left(
0,h\right) }=1$ for any $x,h\in \mathbb{R}$. The binomial theorem for the
rising factorial shows that (\ref{Polya probabilities}) still defines a
distribution $X_n^{a,b,c}$ for real values of the parameters $a,b\ge0$ and $%
c $, satisfying the compatibility condition (\ref{hypotheses on a,b,c}).

It is known (see \cite{Kozniewska}) that the partial first absolute centered
moment of the P\'{o}lya distribution $X_{n}^{x,1-x,c}$ is given by%
\begin{equation}
E\left( \left( nx-X_{n}^{x,1-x,c}\right) 1_{X_{n}^{x,1-x,c}\leq s-1}\right)
=\sum_{k=0}^{s-1}\left( nx-k\right)
p_{n,k}^{x,1-x,c}=sp_{n,s}^{x,1-x,c}\left( 1-x+\left( n-s\right) c\right)
\label{Polya partial first centered moments}
\end{equation}%
for all $s\in \left\{ 1,\ldots ,n\right\} $.

Denoting by $\mathcal{F}\left( \left[ 0,1\right] \right) $ the set of
real-valued functions defined on $\left[ 0,1\right] $, for an integer $n\geq
1$ denote by $P_{n}^{a,b,c}:\mathcal{F}\left( \left[ 0,1\right] \right)
\rightarrow \mathcal{F}\left( \left[ 0,1\right] \right) $ the operator
defined by%
\begin{equation}
P_{n}^{a,b,c}\left( f;x\right) =E\left( \frac{1}{n}X_{n}^{a,b,c}\right)
=\sum_{k=0}^{n}f\left( \frac{k}{n}\right) p_{n}^{a,b,c}
\label{general operator}
\end{equation}%
where the parameters $a,b\geq 0$ and $c\geq -\frac{\min \left\{
x,1-x\right\} }{n-1}$ may depend on $n$ and $x\in \left[ 0,1\right] $ and
satisfy the compatibility condition (\ref{hypotheses on a,b,c}) (see \cite%
{PaPaTr1}). In particular, consider the Bernstein operator%
\begin{equation}
B_{n}\left( f;x\right) =P_{n}^{x,1-x,0}\left( f;x\right) ,
\label{Bernstein operator}
\end{equation}%
the Bernstein-Stancu operator%
\begin{equation}
P_{n}^{c}\left( f;x\right) =P_{n}^{x,1-x,c}\left( f;x\right) ,
\label{Bernstein-Stancu operator}
\end{equation}%
and the operator $R_{n}$ introduced by the last two authors in \cite{PaPaTr1}%
\begin{equation}
R_{n}\left( f;x\right) =P_{n}^{x,1-x,-\min \left\{ x,1-x\right\} /\left(
n-1\right) },  \label{Pascu operator}
\end{equation}%
which corresponds to the minimal value of the replacement parameter $c$
(satisfying the compatibility condition (\ref{hypotheses on a,b,c})) in the
case $a=x$ and $b=1-x$.

Finally, recall that a random variable $X$ is said to be smaller in the
convex order than a random variable $Y$ (denoted by $X\le_{\text{cx}}Y$) iff
\begin{equation}
E \left(\phi(X)\right) \le E \left(\phi (Y)\right)  \label{convex ordering}
\end{equation}
for any convex function $\phi:\mathbb{R} \rightarrow \mathbb{R}$ for which
the above expectations exist.

If $X$ and $Y$ are random variables for which the means $EX$, $EY$ exist and
are equal, it is known (see e.g. \cite{Shanthikumar}, Theorem 3.A.1) that $%
X\leq _{\text{cx}}Y$ iff
\begin{equation}
E\left( (t-X)_{+}\right) \leq E\left( (t-Y)_{+}\right) ,\qquad \text{for all
}t\in \mathbb{R},  \label{convex ordering biss}
\end{equation}%
where $x_{+}=\max \{x,0\}$ denotes the positive part of $x\in \mathbb{R}$.

\section{Main results\label{Section Main results}}

In order to prove the main result of this section, we begin with the
following auxiliary result, of independent interest.

\begin{lemma}
\label{interlacing lemma}For any real number $x\in \left( 0,1\right) $,
integers $n\geq 3$ and $k\in \left\{ 1,\ldots ,n-2\right\} $, there exists a
disjoint partition $\left\{ n_{1},\ldots ,n_{k}\right\} \bigsqcup \left\{
m_{1},\ldots ,m_{n-k-1}\right\} $ of the set $\left\{ 1,2,\ldots
,n-1\right\} $ such that
\begin{equation}
n_{i}\leq \frac{i}{x},\quad i\in \left\{ 1,\ldots ,k\right\}\qquad \text{and} \qquad m_{i}\leq \frac{i}{1-x},\quad i\in \left\{ 1,\ldots ,n-k-1\right\} .
\end{equation}
\end{lemma}

\begin{proof}
First note that for $i\geq 1$ the interval $\left( i,i+1\right) $ cannot
contain two distinct elements of the set $\left\{ \frac{1}{x},\ldots ,%
\frac{k}{x}\right\} $. This is so for otherwise, there would exist indices $%
1\leq j_{1}<j_{2}\leq k$ such that $i<\frac{j_{1}}{x}<\frac{j_{2}}{x}<i+1$, and therefore%
\[
\frac{1}{x}\leq \frac{j_{2}-j_{1}}{x}<i+1-i=1,
\]%
which implies $x>1$, a contradiction. A similar proof shows that the
interval $\left( i,i+1\right) $ cannot contain two distinct elements of the
set $\left\{ \frac{1}{1-x},\ldots ,\frac{n-k-1}{1-x}\right\} $.

Secondly, note that the interval $\left( i,i+1\right) $ cannot contain an
element of the set $\left\{ \frac{1}{x},\ldots ,\frac{k}{x}\right\} $ and
an element of the set $\left\{ \frac{1}{1-x},\ldots ,\frac{n-k-1}{1-x}%
\right\} $. This is so for otherwise there would exist indices $1\leq
j_{1}\leq k$ and $1\leq j_{2}\leq n-k-1$ such that $i<\frac{j_{1}}{x}<i+1$ and $i<\frac{j_{2}}{1-x}<i+1$, and therefore%
\[
i=ix+i\left( 1-x\right) <j_{1}+j_{2}<\left( i+1\right) x+\left( i+1\right)
\left( 1-x\right) =i+1,
\]%
which is impossible since $j_{1}+j_{2}$ is an integer.

A similar proof shows that if the interval $\left[ i,i+1\right] $ contains an element of the set $\left\{ \frac{1}{x},\ldots ,\frac{k}{x}\right\} $ and
an element of the set $\left\{ \frac{1}{1-x},\ldots ,\frac{n-k-1}{1-x}%
\right\} $, then they both must be equal either to $i$ or to $i+1$. Note that in the
latter case the interval $(i+1,i+2]$ cannot contain any element of the
set $\left\{ \frac{1}{x},\ldots ,\frac{k}{x}\right\} \cup \left\{
\frac{1}{1-x},\ldots ,\frac{n-k-1}{1-x}\right\} $ (since $x<1$ and $1-x<1$).

Define the sequences $\left( n_{i}^{\prime }\right) _{i=\overline{1,k}}$
and $\left( m_{i}^{\prime }\right ) _{i=\overline{1,n-k-1}}$ by $n_i^\prime=j$ if $\frac{i}{x}\in (j,j+1]$ for some $j\in \mathbb{N}$, $i\in \left\{ 1,\ldots ,k\right\}$, and%
\[
m_{i}^{\prime }=\left\{
\begin{array}{ll}
j, & \text{if }\frac{i}{1-x}\in (j,j+1)\text{ for some }j\in \mathbb{N} \\
j+1, & \text{if }\frac{i}{1-x}=j+1\text{ for some }j\in \mathbb{N}%
\end{array}%
\right. ,\qquad i\in \left\{ 1,\ldots ,n-k-1\right\} ,
\]%
thus $n_{i}^{\prime }\leq \frac{i}{x}$ for $i\in \left\{ 1,\ldots
,k\right\} $ and $m_{i}^{\prime }\leq \frac{i}{1-x}$ for $i\in \left\{
1,\ldots ,n-k-1\right\} $.

The first part of the proof shows that $n_{1}^{\prime },\ldots n_{k}^{\prime }, m_{1}^{\prime}, \ldots,m_{n-k-1}^{\prime}$ are all distinct, and in particular this
shows that the set $N= \left\{
n_{1}^{\prime },\ldots n_{k}^{\prime }, m_{1}^{\prime}, \ldots,m_{n-k-1}^{\prime}\right\}$ has $n-1$ elements.

Denoting by $r$ the number of distinct elements of the set $N \cap \left( n-1,\infty \right) $, it
follows that the set $\left\{ 1,\ldots ,n-1\right\} \backslash \left(
N \cap \lbrack 0,n-1]\right) $ also
has $r$ elements, and therefore there exists a bijection $f:N \cap \left( n-1,\infty \right)
\rightarrow \left\{ 1,\ldots ,n-1\right\} \backslash \left( N \cap \lbrack 0,n-1]\right) $.

Define
\[
n_{i}=\left\{
\begin{array}{ll}
n_{i}, & \text{if }n_{i}^{\prime }\leq n-1 \\
f\left( n_{i}^{\prime }\right) , & \text{if }n_{i}^{^{\prime }}>n-1%
\end{array}%
\right. ,\qquad i\in\{1,\ldots,k\},
\]%
and%
\[
m_{i}=\left\{
\begin{array}{ll}
m_{i}^{\prime }, & \text{if }m_{i}^{\prime }\leq n-1 \\
f\left( m_{i}^{\prime }\right) , & \text{if }m_{i}^{\prime }>n-1%
\end{array}%
\right. ,\qquad i\in \left\{ 1,\ldots ,n-k-1\right\} .
\]

It is not difficult to see that $n_1,\ldots,n_k,m_1,\ldots,m_{n-k-1}$ are distinct, $\left\{ 1,\ldots ,n-1\right\} =\left\{ n_{1},\ldots
,n_{k},m_{1},\ldots ,m_{n-k-1}\right\} $
(recall the definition of the function $f,$ in particular $f\leq n-1$), and
they satisfy
\[
n_{i}\leq \frac{i}{x},\text{\quad }i\in \left\{ 1,\ldots ,k\right\} \qquad \text{and} \qquad
m_{i}\leq \frac{i}{1-x},\text{\quad }i\in \left\{ 1,\ldots ,n-k-1\right\} ,
\]%
concluding the proof.
\end{proof}

A second auxiliary result of independent interest is the following
monotonicity of the (partial) first centered moment of P\'{o}lya random
variables.

\begin{lemma}
\label{monotonicity of Polya first centered moment} For any $x\in \left[ 0,1%
\right] $ and any integers $n\geq 1$ and $k\in \left\{ 0,1,\ldots ,n\right\}
$, the sum%
\begin{equation}
\sum_{i=0}^{k}\left( x-\frac{i}{n}\right) p_{n,i}^{x,1-x,c}
\label{claim of lemma}
\end{equation}%
is a non-decreasing function of $c\geq -\frac{1}{n-1}\min \{x,1-x\}$.

Moreover, for $x\in \left( 0,1\right) $, $n\geq 2,$ and $k\in \left\{
0,\ldots ,n-1\right\} $, the above sum is increasing with respect to $c\geq -%
\frac{1}{n-1}\min \{x,1-x\}$.
\end{lemma}

\begin{proof}
For $k=n$ we have
\begin{equation}
\sum_{i=0}^{n}\left( x-\frac{i}{n}\right) p_{n,i}^{x,1-x,c}=x-\frac{1}{n}%
EX_{n}^{x,1-x,c}=x-x=0,
\end{equation}%
and for $k=0$ we have
\begin{equation}
\sum_{i=0}^{0}\left( x-\frac{i}{n}\right)
p_{n,i}^{x,1-x,c}=xp_{n,0}^{x,1-x,c}=x\left( 1-x\right) \left( 1-x+c\right)
\cdot \ldots \cdot \left( 1-x+(n-1)c\right) ,
\end{equation}%
thus the claim of the lemma is true in these cases.

Also, since for $x=0$ and $x=1$ the probabilities $p_{n,i}^{x,1-x,c}$ are
independent on the value of $c$ ($p_{n,0}^{0,1,c}=p_{n,n}^{1,0,c}=1$, and $%
p_{n,i}^{x,1-x,c}=0$ for $i\in \left\{ 1,\ldots ,n-1\right\} $), the claim
of the lemma is also true in these cases.

Without loss of generality we may therefore assume that $x\in (0,1)$, $k\in
\{1,\ldots ,n-1\}$, and $n\geq 2$. Since the expression in (\ref{claim of
lemma}) is a continuous, differentiable function of $c\geq -\frac{1}{n-1}%
\min \{x,1-x\}$, in order to prove the claim of the lemma, using (\ref{Polya
partial first centered moments}), it suffices to show that for any $c>-\frac{1}{n-1}\min \{x,1-x\}$ we have
\begin{equation}
\frac{\partial }{\partial c}\sum_{i=0}^{k}\left( nx-i\right)
p_{n,i}^{x,1-x,c}=\frac{\partial }{\partial c}\left( \left( k+1\right)
p_{n,k+1}^{x,1-x,c}\left( 1-x+\left( n-\left( k+1\right) \right) c\right)
\right) >0.
\end{equation}%

Taking logarithms and using (\ref{Polya probabilities}), we have left to
prove that%
\[
\frac{\partial }{\partial c}\left( \ln \left( \left( k+1\right) C_{n}^{k+1}%
\frac{x^{\left( k+1,c\right) }\left( 1-x\right) ^{\left( n-k-1,c\right)
}\left( 1-x+\left( n-k-1\right) c\right) }{1^{\left( n,c\right) }}\right)
\right) >0,
\]%
or equivalent%
\begin{equation}
\sum_{i=0}^{k}\frac{i}{x+ic}+\sum_{i=0}^{n-k-1}\frac{i}{1-x+ic}%
>\sum_{i=0}^{n-1}\frac{i}{1+ic},  \label{claim 1}
\end{equation}%
for any $c>-\frac{1}{n-1}\min \{x,1-x\}$.

For $k=n-1$ the above inequality is readily satisfied (recall that $x\in
\left( 0,1\right) $), so we have left to consider the case $n\geq 3$ and $%
k\in \left\{ 1,\ldots ,n-2\right\} $.

For arbitrarily fixed $n\geq 3$, $k\in \left\{ 1,\ldots ,n-2\right\} $, $%
x\in \left( 0,1\right) $, and $c>-\frac{1}{n-1}\min \{x,1-x\}$, the function
$\varphi :\left[ 1,\frac{n-1}{\min \left\{ x,1-x\right\} }\right]
\rightarrow \mathbb{R}$ defined by $\varphi \left( u\right) =\frac{u}{1+uc}$ is increasing.

Lemma \ref{interlacing lemma} shows that we can find a partition $\left\{
n_{1},\ldots ,n_{k}\right\} \bigsqcup \left\{ m_{1},\ldots
,m_{n-k-1}\right\} $ of the set $\left\{ 1,2,\ldots ,n-1\right\} $ such that
$n_{i}\leq \frac{i}{x}$ for $i\in \left\{ 1,\ldots ,k\right\} $ and $%
m_{i}\leq \frac{i}{1-x}$ for $i\in \left\{ 1,\ldots ,n-k-1\right\} $, and
therefore we obtain%
\[
\sum_{i=1}^{k}\varphi \left( \frac{i}{x}\right) +\sum_{i=1}^{n-k-1}\varphi
\left( \frac{i}{1-x}\right) >\sum_{i=1}^{k}\varphi \left( n_{i}\right)
+\sum_{i=1}^{n-k-1}\varphi \left( m_{i}\right) =\sum_{i=1}^{n-1}\varphi
\left( i\right) ,
\]%
which is equivalent to (\ref{claim 1}), thus concluding the proof.
\end{proof}

With the above preparation we can now proceed to prove the convex ordering
of P\'{o}lya random variables $X_{n}^{a,b,c}$ with respect to the
replacement parameter $c$. The precise statement is the following.

\begin{theorem}
\label{Theorem on Polya convex ordering}The P\'{o}lya random variables $%
X_{n}^{x,1-x,c}$ satisfy the following convex ordering%
\begin{equation}
X_{n}^{x,1-x,c}\leq _{\text{cx}}X_{n}^{x,1-x,c^{\prime }},
\label{convex ordering for Polya}
\end{equation}%
for any integer $n\geq 1$, $x\in \left[ 0,1\right] $, and any $c^{\prime
}\geq c\geq -\frac{1}{n-1}\min \left\{ x,1-x\right\} $.
\end{theorem}

\begin{proof}
Since $EX_{n}^{x,1-x,c}=nx$ for any value of $c$ satisfying the
compatibility condition $c\geq -\frac{1}{n-1}\min \left\{ x,1-x\right\} $,
the claim of the theorem is equivalent by (\ref{convex ordering biss}) to
\begin{equation}
\frac{\partial }{\partial c}E\left( \left( t-X_{n}^{x,1-x,c}\right)
_{+}\right) \geq 0,\qquad t\in \mathbb{R}.
\end{equation}

Since $X_{n}^{x,1-x,c}$ takes values in $\left\{ 0,1,\ldots ,n\right\} $ and
$EX_{n}^{x,1-x,c}=nx$ is independent of $c$, it is easy to see that the
above inequality is satisfied for $t<0$ and $t>n$, thus it remains to prove
it for $t\in \lbrack 0,n]$.

Replacing $t$ by $n t$ with $t\in[0,1]$, the above inequality is thus
equivalent to
\begin{equation}
\frac{\partial }{\partial c}E\left( t-\frac{1}{n}X_{n}^{x,1-x,c}\right)
_{+}=\sum_{i=0}^{\left[ nt\right] }\left( \left( t-\frac{i}{n}\right) \frac{%
\partial }{\partial c}p_{n,i}^{x,1-x,c}\right) \geq 0,\qquad t\in \left[ 0,1%
\right] .  \label{aux 6}
\end{equation}

The above inequality is further equivalent to the apparent weaker inequality%
\begin{equation}  \label{aux 7}
\sum_{i=0}^{k}\left( \left( \frac{k}{n}-\frac{i}{n}\right) \frac{\partial }{%
\partial c}p_{n,i}^{x,1-x,c}\right) \geq 0,\qquad k\in \left\{ 0,1,\ldots
,n\right\} ,
\end{equation}%
the reason being the following.

For $t=1=\frac{n}{n}$ the inequality (\ref{aux 6}) follows immediately from (%
\ref{aux 7}). For $t\in \lbrack 0,1)$, we can write $t=\alpha \frac{k}{n}%
+\left( 1-\alpha \right) \frac{k+1}{n}$ as a convex combination of $\frac{k}{%
n}$ and $\frac{k+1}{n}$, where $k=\left[ nt\right] $ and $\alpha =k+1-nt\in
\lbrack 0,1)$. If the inequality (\ref{aux 7}) holds true, then%
\begin{eqnarray*}
&&\frac{\partial }{\partial c}\sum_{i=0}^{[nt]}\left( \left( t-\frac{i}{n}%
\right) p_{n,i}^{x,1-x,c}\right) \\
&=&\frac{\partial }{\partial c}\left(
\alpha \sum_{i=0}^{k}\left( \left( \frac{k}{n}-\frac{i}{n}\right)
p_{n,i}^{x,1-x,c}\right) +\left( 1-\alpha \right) \sum_{i=0}^{k}\left(
\left( \frac{k+1}{n}-\frac{i}{n}\right) p_{n,i}^{x,1-x,c}\right) \right) \\
&=&\frac{\partial }{\partial c}\left( \alpha \sum_{i=0}^{k}\left( \left(
\frac{k}{n}-\frac{i}{n}\right) p_{n,i}^{x,1-x,c}\right) +\left( 1-\alpha
\right) \sum_{i=0}^{k+1}\left( \left( \frac{k+1}{n}-\frac{i}{n}\right)
p_{n,i}^{x,1-x,c}\right) \right) \\
&=&\alpha \sum_{i=0}^{k}\left( \left( \frac{k}{n}-\frac{i}{n}\right) \frac{%
\partial p_{n,i}^{x,1-x,c}}{\partial c}\right) +\left( 1-\alpha \right)
\sum_{i=0}^{k+1}\left( \left( \frac{k+1}{n}-\frac{i}{n}\right) \frac{%
\partial p_{n,i}^{x,1-x,c}}{\partial c}\right) \\
&\geq &0,
\end{eqnarray*}%
by Lemma \ref{monotonicity of Polya first centered moment}, thus proving (%
\ref{aux 6}).

In order to prove (\ref{aux 7}), first note that the claim is trivial if $%
x=1 $, for in this case $p_{n,n}^{1,0,c}=1$ and $p_{n,i}^{1,0,c}=0$ for $%
i\in \left\{ 0,\ldots ,n-1\right\} $ (thus $p_{n,i}^{1,0,c}$ is independent
of the value of the replacement parameter $c$). Without loss of generality
we can therefore assume that $x\neq 1$.

Suppose that (\ref{aux 7}) does not hold for a certain $x\in[0,1)$ and $k\in
\{0,1,\ldots, n\}$, that is
\begin{equation}  \label{aux 10}
\sum_{i=0}^{k}\left( \left( \frac{k}{n}-\frac{i}{n}\right) \frac{\partial }{%
\partial c}p_{n,i}^{x,1-x,c}\right) < 0.
\end{equation}

We distinguish the following cases.

\begin{itemize}
\item[i)] $nx>k$

Note that in this case we cannot have $k=0$, for in this case the sum in (%
\ref{aux 10}) is equal to $0$. Since $k\geq 1$ and $x\in[0,1)$ we can write $%
x=\alpha \frac{k-1}{n}+\left( 1-\alpha \right) \frac{k}{n}$, where $\alpha
=k-nx<0$.

Using again Lemma \ref{monotonicity of Polya first centered moment} and the
same argument as above we obtain%
\begin{eqnarray*}
&&\alpha \sum_{i=0}^{k-1}\left( \left( \frac{k-1}{n}-\frac{i}{n}\right)
\frac{\partial p_{n,i}^{x,1-x,c}}{\partial c}\right) +\left( 1-\alpha
\right) \sum_{i=0}^{k}\left( \left( \frac{k}{n}-\frac{i}{n}\right) \frac{%
\partial p_{n,i}^{x,1-x,c}}{\partial c}\right) \\
&=&\alpha \sum_{i=0}^{k-1}\left( \left( \frac{k-1}{n}-\frac{i}{n}\right)
\frac{\partial p_{n,i}^{x,1-x,c}}{\partial c}\right) +\left( 1-\alpha
\right) \sum_{i=0}^{k-1}\left( \left( \frac{k}{n}-\frac{i}{n}\right) \frac{%
\partial p_{n,i}^{x,1-x,c}}{\partial c}\right) \\
&=&\frac{\partial }{\partial c}\sum_{i=0}^{k-1}\left( \left( \alpha \left(%
\frac{k-1}{n}-\frac{i}{n}\right)+\left( 1-\alpha \right)\left( \frac{k}{n}-%
\frac{i}{n}\right)\right) p_{n,i}^{x,1-x,c}\right) \\
&=&\sum_{i=0}^{k-1}\left( \left( x-\frac{i}{n}\right) \frac{\partial
p_{n,i}^{x,1-x,c}}{\partial c}\right) \\
&\geq &0,
\end{eqnarray*}%
and therefore (the second sum on the first line above being assumed to be
strictly negative, and since $\alpha <0$) we deduce that%
\begin{equation}  \label{aux 11}
\sum_{i=0}^{k-1}\left( \left( \frac{k-1}{n}-\frac{i}{n}\right) \frac{%
\partial p_{n,i}^{x,1-x,c}}{\partial c}\right) <0.
\end{equation}

We showed that if the sum in (\ref{aux 10}) is strictly negative, then so is
the sum in (\ref{aux 11}). Proceeding inductively on $k$, this implies that%
\[
0\equiv \sum_{i=0}^{0}\left( \left( \frac{0}{n}-\frac{i}{n}\right) \frac{%
\partial p_{n,i}^{x,1-x,c}}{\partial c}\right) <0,
\]%
a contradiction.

\item[ii)] $nx\leq k < k+1$

Note that in this case we cannot have $k=n$, for
\begin{equation}
\sum_{i=0}^{n}\left( \left( \frac{n}{n}-\frac{i}{n}\right) \frac{\partial }{%
\partial c}p_{n,i}^{x,1-x,c}\right) =\frac{\partial}{\partial c} E \left( 1
- \frac{1}{n}X_{n}^{x,1-x,c}\right)= \frac{\partial}{\partial c} \left(1-
x\right)\equiv 0.
\end{equation}

Since $k\leq n-1$ and $x\in[0,1)$ we can write $x=\alpha \frac{k}{n}+\left(
1-\alpha \right) \frac{k+1}{n}$, where $k+1\leq n$ and $\alpha =k+1-nx\geq 1
$.

We have%
\begin{eqnarray*}
&&\alpha \sum_{i=0}^{k}\left( \left( \frac{k}{n}-\frac{i}{n}\right) \frac{%
\partial p_{n,i}^{x,1-x,c}}{\partial c}\right) +\left( 1-\alpha \right)
\sum_{i=0}^{k+1}\left( \left( \frac{k+1}{n}-\frac{i}{n}\right) \frac{%
\partial p_{n,i}^{x,1-x,c}}{\partial c}\right) \\
&=&\alpha \sum_{i=0}^{k}\left( \left( \frac{k}{n}-\frac{i}{n}\right) \frac{%
\partial p_{n,i}^{x,1-x,c}}{\partial c}\right) +\left( 1-\alpha \right)
\sum_{i=0}^{k}\left( \left( \frac{k+1}{n}-\frac{i}{n}\right) \frac{\partial
p_{n,i}^{x,1-x,c}}{\partial c}\right) \\
&=&\frac{\partial }{\partial c}\sum_{i=0}^{k}\left( \left( \alpha \frac{k}{n}%
+\left( 1-\alpha \right) \frac{k+1}{n}-\frac{i}{n}\right)
p_{n,i}^{x,1-x,c}\right) \\
&=&\frac{\partial }{\partial c}\sum_{i=0}^{k}\left( \left( x-\frac{i}{n}%
\right) p_{n,i}^{x,1-x,c}\right) \\
&\geq &0
\end{eqnarray*}%
by Lemma \ref{monotonicity of Polya first centered moment}.

Using (\ref{aux 10}) and the fact that $\alpha \geq 1$, from the above
inequality we conclude that
\begin{equation}
\sum_{i=0}^{k+1}\left( \left( \frac{k+1}{n}-\frac{i}{n}\right) \frac{%
\partial p_{n,i}^{x,1-x,c}}{\partial c}\right) <0.  \label{aux 12}
\end{equation}

We showed that if the sum in (\ref{aux 10}) is strictly negative, then so is
the sum in (\ref{aux 12}). Proceeding inductively on $k$, we obtain%
\[
0\equiv \frac{\partial }{\partial c}\left( 1-x\right) =\frac{\partial }{%
\partial c}E\left( 1-\frac{1}{n}X_{n}^{x,1-x,c}\right) =\sum_{i=0}^{n}\left(
\left( \frac{n}{n}-\frac{i}{n}\right) \frac{\partial p_{n,i}^{x,1-x,c}}{%
\partial c}\right) <0,
\]%
a contradiction.
\end{itemize}

In both cases above we showed that (\ref{aux 10}) leads to a contradiction.
This shows that (\ref{aux 7}) holds true, thus concluding the proof of the
theorem.
\end{proof}

\section{An application to the error estimate of P\'{o}lya-Stancu operators
\label{Section Applications}}

As an application of Theorem \ref{Theorem on Polya convex ordering}, we have
the following.

\begin{theorem}
\label{Theorem on monotonicty of error}For any convex function $f:\left[ 0,1%
\right] \rightarrow \mathbb{R}$ the absolute value of the error of
approximation of the P\'{o}lya-Stancu operator $P_{n}^{c}$ is a
non-decreasing function of $c$, that is
\begin{equation}
\left\vert P_{n}^{c_{2}}f\left( x\right) -f\left( x\right) \right\vert \geq
\left\vert P_{n}^{c_{1}}f\left( x\right) -f\left( x\right) \right\vert ,
\label{monotonicity of error estimate}
\end{equation}%
for any integer $n\geq 2$, $x\in \left[ 0,1\right] $, and any $%
c_{2}>c_{1}\geq -\frac{1}{n-1}\min \left\{ x,1-x\right\} $.

If moreover%
\begin{equation}
B_{n}f\left( x\right) \neq B_{1}f\left( x\right)
\end{equation}%
for certain values of $n\geq 2$ and $x\in \left[ 0,1\right] $, then the
above monotonicity is strict, that is
\begin{equation}
\left\vert P_{n}^{c_{2}}f\left( x\right) -f\left( x\right) \right\vert
>\left\vert P_{n}^{c_{1}}f\left( x\right) -f\left( x\right) \right\vert ,
\end{equation}%
for any $c_{2}>c_{1}\geq -\frac{1}{n-1}\min \left\{ x,1-x\right\} $.
\end{theorem}

\begin{proof}
Jensen's inequality shows that%
\[
P_{n}^{c}f\left( x\right) =Ef\left( \frac{1}{n}X_{n}^{x,1-x,c}\right) \geq
f\left( \frac{1}{n}EX_{n}^{x,1-x,c}\right) =f\left( x\right) ,\qquad c\geq -%
\frac{1}{n-1}\min \left\{ x,1-x\right\} ,
\]%
thus the claim (\ref{monotonicity of error estimate}) is equivalent to $P_{n}^{c_{2}}f\left( x\right) \geq P_{n}^{c_{1}}f\left( x\right)$,
and it follows immediately from Theorem \ref{Theorem on Polya convex
ordering} and the definition (\ref{convex ordering}) of convex ordering.

To prove the second part of the theorem, note that if for certain values $n\geq 2$ and $%
x\in \left[ 0,1\right] $ we have $P_{n}^{c_{2}}f\left( x\right) =P_{n}^{c_{1}}f\left( x\right)$
for some $c_{2}>c_{1}\geq -\frac{1}{n-1}\min \left\{ x,1-x\right\} $, then by
the first part of the proof we have that
\begin{equation}
P_{n}^{c}f\left( x\right) =C,\qquad c\in \left[ c_{1},c_{2}\right] ,
\label{aux 15}
\end{equation}
is a constant function of $c$ (the constant $C$ may depend on $n$ and $x$).

The above still holds if we replace $f$ by $f+1+\max_{x\in \left[ 0,1\right]
}f\left( x\right) $ (being convex, $f$ is also continuous), thus without
loss of generality we may assume that $f\left( x\right) >0$ for $x\in \left[
0,1\right] $.

Next note that for $x\in \left( 0,1\right) $, from the definition (\ref%
{rising factorial}) of the rising factorial, it follows that $x^{\left(
k,c\right) }\left( 1-x\right) ^{\left( n-k,c\right) }$ is a polynomial of
degree $n-2$ in the variable $c$ if $k\in \left\{ 1,\ldots ,n-1\right\} $
(with leading coefficient $x\left( 1-x\right) \left( k-1\right) !\left(
n-k-1\right) !$), and a polynomial of degree $n-1$ in the variable $c$ if $%
k\in \left\{ 0,n\right\} $ (with leading coefficient $\left( n-1\right)
!\left( 1-x\right) $ for $k=0$ and $\left( n-1\right) !x$ for $k=n$). Using
this and the definition (\ref{Bernstein-Stancu operator}) of the operator $P_{n}^{c}$%
, it can be seen that for fixed values $n\geq 2$, $x\in \left[ 0,1\right] $
and $f$,
\begin{eqnarray*}
P_{n}^{c}f\left( x\right) &=&\frac{\sum_{k=0}^{n}\left( f\left( \frac{k}{n}\right) C_{n}^{k}x^{\left(
k,c\right) }\left( 1-x\right) ^{\left( n-k,c\right) }\right) }{1^{\left(
n,c\right) }} \\
&=&\frac{\left( \left( n-1\right) !f\left( 0\right) \left( 1-x\right)
+\left( n-1\right) !f\left( 1\right) x\right) c^{n-1}+\ldots
+\sum_{k=0}^{n}f\left( \frac{k}{n}\right) x^{k}\left( 1-x\right) ^{n-k}}{%
\left( n-1\right) !c^{n-1}+\ldots +1}
\end{eqnarray*}%
is the ratio of two polynomials of degree $n-1$ in the variable $c\geq -%
\frac{1}{n-1}\min \left\{ x,1-x\right\} $ (recall that $f>0$ and the special
cases $k=0$ and $k=n$ above).

From (\ref{aux 15}) we conclude that these two polynomials (in the variable $%
c$) are a constant multiple of each other; in particular this shows that
their leading and free term coefficients are proportional, thus%
\[
\frac{\left( \left( n-1\right) !f\left( 0\right) \left( 1-x\right) +\left(
n-1\right) !f\left( 1\right) x\right) }{\left( n-1\right) !}=\frac{%
\sum_{k=0}^{n}f\left( \frac{k}{n}\right) x^{k}\left( 1-x\right) ^{n-k}}{1},
\]%
or equivalent%
\[
f\left( 0\right) \left( 1-x\right) +f\left( 1\right) x=\sum_{k=0}^{n}f\left(
\frac{k}{n}\right) x^{k}\left( 1-x\right) ^{n-k}.
\]

Finally, note that $\sum_{k=0}^{n}f\left( \frac{k}{n}\right) x^{k}\left(
1-x\right) ^{n-k}=B_{n}f\left( x\right) $ is just the Bernstein polynomial
of degree $n$ corresponding to $f$ evaluated at $x$, and $f\left( 0\right)
\left( 1-x\right) +f\left( 1\right) x=B_{1}f\left( x\right) $, thus we have
equivalent%
\begin{equation}
B_{n}f\left( x\right) =B_{1}f\left( x\right) .  \label{aux 16}
\end{equation}

We have shown that (\ref{aux 15}) implies the condition (\ref{aux 16}).
Therefore if (\ref{aux 16}) does not hold, then (\ref{aux 15}) cannot hold, thus concluding the proof of the
theorem.
\end{proof}

\end{document}